\pdfoutput=1
\RequirePackage{ifpdf}
\ifpdf % We~are running pdfTeX in pdf mode
\documentclass[pdftex]{sigma}
\else
\documentclass{sigma}
\fi

\usepackage{tensor}

\numberwithin{equation}{section}

\newtheorem{Theorem}{Theorem}[section]
\newtheorem*{Theorem*}{Theorem}
\newtheorem{Corollary}[Theorem]{Corollary}
\newtheorem{Lemma}[Theorem]{Lemma}
\newtheorem{Proposition}[Theorem]{Proposition}
\newtheorem{Conjecture}[Theorem]{Conjecture}
 { \theoremstyle{definition}
\newtheorem{Definition}[Theorem]{Definition}

}

\begin{document}
%\allowdisplaybreaks

\newcommand{\arXivNumber}{2403.06782}

\renewcommand{\PaperNumber}{018}

\FirstPageHeading

\ShortArticleName{Mass from an Extrinsic Point of View}

\ArticleName{Mass from an Extrinsic Point of View}

\Author{Alexandre DE~SOUSA~$^{\rm a}$ and Frederico GIR\~AO~$^{\rm b}$}

\AuthorNameForHeading{A.~de~Sousa and F.~Gir\~ao}

\Address{$^{\rm a)}$~Escola de Ensino Fundamental e M\'edio Santa Luzia, Fortaleza, 60110-300, Brazil}
\EmailD{\href{mailto:alexandre.ads@pm.me}{alexandre.ads@pm.me}}

\Address{$^{\rm b)}$~Departamento de Matem\'atica, Universidade Federal do Cear\'a, Fortaleza, 60455-760, Brazil}
\EmailD{\href{mailto:fred@mat.ufc.br}{fred@mat.ufc.br}}

\ArticleDates{Received October 22, 2024, in final form March 03, 2025; Published online March 18, 2025}

\Abstract{We express the $q$-th Gauss--Bonnet--Chern mass of an immersed submanifold of Euclidean space as a linear combination of two terms: the total $(2q)$-th mean curvature and the integral, over the entire manifold, of the inner product between the $(2q+1)$-th mean curvature vector and the position vector of the immersion. As a consequence, we obtain, for each $q$, a geometric inequality that holds whenever the positive mass theorem (for the $q$-th Gauss--Bonnet--Chern mass) holds.}


\Classification{83C99; 53C40; 51M16}

\section{Introduction}

In this article, we explore the concept of mass from an extrinsic point of view.
Our approach is based on the introduction of a class of immersions
in Euclidean space whose members we call \emph{asymptotically Euclidean immersions} (Definition \ref{defn AE}) and on an integral
identity (Theorem \ref{main theorem}) that can be seen as a version of the
classical Hsiung-Minkowski formulas \cite{Hsiungintegralformulasclosed1956,KwongExtensionHsiungMinkowski2016} for a class of non-compact
immersed submanifolds. Through this identity, we deduce a geometric inequality
(Corollary \ref{corollary main theorem}) that must be satisfied whenever the
positive mass conjecture for the GBC mass (and ADM mass, in particular) is valid.

We also explore the necessary conditions for vector fields to generate
the asymptotic charges of interest (Proposition \ref{mass using another vector field})
and present two conjectures. In one of them, we conjecture that asymptotically Euclidean
spaces admit an asymptotically Euclidean isometric immersion (Conjecture
\ref{q1}); in the other, we conjecture that the aforementioned geometric inequality must hold
whenever an asymptotically Euclidean immersion satisfies a natural hypothesis
(Conjecture \ref{q2}). If both conjectures are valid, the positive mass
conjecture for the GBC mass (and ADM mass, in particular) can be directly deduced
from our results.

\section{Mass of asymptotically Euclidean spaces}

In this section, we establish the background needed for understanding
our method. As far as we are aware, the results presented
in Proposition \ref{mass using another vector field} and Corollary
\ref{gradient field} are novel in the literature. They are fundamental ingredients
for the employment of our line of action, since they establish sufficient conditions
for a vector field to generate the asymptotic charges we are interest in.

We begin by remembering the concept of an asymptotically Euclidean end.
\begin{Definition}[asymptotically Euclidean end]
 An asymptotically Euclidean end of order
 $\tau$, with \(\tau > 0\), is a Riemannian
 manifold $(E^n,g)$, $n \geq 3$, for which there exists a diffeomorphism
 $\Psi\colon E \to \mathbb{R}^n \setminus \overline{B}_1(0)$, introducing coordinates
 in $E$, say $\Psi(x) = \bigl(x^1,x^2, \ldots, x^n\bigr)$, such that, in these coordinates,
 the following asymptotic condition holds:
 \begin{equation} \label{decay conditions}
  |g_{ij} - \delta_{ij}| + \rho|g_{ij,k}| + \rho^2 |g_{ij,kl}| = O (\rho^{-\tau}), \qquad
  \text{as}\ \rho \to \infty,
 \end{equation}
 for all $i,j,k,l \in \lbrace 1, 2, \ldots, n \rbrace$,
 where the $g_{ij}$'s are the coefficients of $g$ with respect to the coordinates
 \(\Psi(x)\), $g_{ij,k} = \partial g_{ij} / \partial x^k$,
 $g_{ij,kl} = \partial g_{ij} / \partial x^k \partial x^l$,
 and \(\rho = |\Psi(x)|\) denotes
 the distance function to the origin with respect
 to the Euclidean metric induced on the end.
\end{Definition}

Our model mass concept is the ADM mass of an end $(E,g)$, introduced by Arnowitt,
Deser and Misner in \cite{ArnowittCoordinateInvarianceEnergy1961}.

\begin{Definition}[ADM mass]
 The ADM mass of an asymptotically Euclidean end \((E^n,g)\) is defined by
 \begin{equation} \label{ADM mass expression}
  \mathrm{m_{ADM}}(E,g) = \frac{1}{2(n-1)\omega_{n-1}} \lim_{\rho \to \infty} \int_{S_\rho} (g_{ij,i} - g_{ii,j}) \nu^j \,  {\rm d}S_\rho,
 \end{equation}
 where $\omega_{n-1}$ is the volume of the unit sphere of dimension $n-1$, $S_\rho$ is the Euclidean coordinate sphere of radius $\rho$, ${\rm d}S_\rho$ is the volume form induced on $S_\rho$ by the Euclidean metric, and $\nu $ is the outward pointing unit normal to $S_\rho$ (with respect to the Euclidean metric).
\end{Definition}

It is known that if $\tau > (n-2)/2$ and the scalar curvature of $(E,g)$ is integrable, then the limit~(\ref{ADM mass expression}) exists, is finite, and is a geometric invariant, that is, two coordinate systems satisfying~(\ref{decay conditions}) yield the same value for it \cite{Bartnikmassasymptoticallyflat1986, ChruscielBoundaryConditionsSpatial1986} (see also \cite{Michelgeometricinvariance2011}).

A complete Riemannian manifold $(M^n,g)$, $n \geq 3$, is said to be asymptotically Euclidean of order $\tau$ if there exists a compact subset $K$ of $M$ such that $M \setminus K$ has finitely many connected components and, for any connected component $E$ of $M \setminus K$, it occurs that $(E,g)$ is an asymptotically Euclidean end of order $\tau$.

If $(M^n,g)$ is an asymptotically Euclidean Riemannian manifold of order $\tau > (n-2)/2$ whose scalar curvature is integrable, then its ADM mass, denoted by $\mathrm{m}_{\mathrm{ADM}}(M,g)$, is defined as the sum of the ADM masses of its ends.

One of the most important results in mathematical general relativity is the positive mass theorem (PMT):

\begin{Theorem} %\label{PMT}
   If $(M^n,g)$ is an asymptotically flat Riemannian manifold of order $\tau > (n-2)/2$ whose scalar curvature is nonnegative and integrable, then each of its ends has nonnegative ADM mass. Moreover, if the ADM mass of at least one of its ends is zero, then $(M,g)$ is isometric to the Euclidean space $(\mathbb{R}^n,\delta)$.
\end{Theorem}

The PMT was settled by Schoen and Yau when $n \leq 7$ \cite{Schoencompletemanifoldswithnonnegative1979, Schoenproofpositivemass1979, Schoentheenergyandthelinearmomentum1981} and when $(M,g)$ is conformally Euclidean \cite{SchoenConformallyflatmanifolds1988}, and by Witten when $M$ is spin \cite{Wittennewproofpositive1981} (see also \cite{ParkerWittenproofpositive1982}). The general cases of the PMT were treated by Lohkamp \cite{LohkampHigherDimensionalPositive2006,Lohkampskinstructuresonminimal2015, Lohkamphyperbolicgeometryandpotentialtheoryminimal2015, Lohkampskinstructuresinscalarcurvature2015} and by Schoen and Yau \cite{Schoenpositivescalarcurvatureandminimal2022}. Proofs for the case when $(M,g)$ is an Euclidean graph (without the rigidity statement) were given by Lam~\mbox{\cite{LamGraphsCasesRiemannian2010,LamGraphCasesRiemannian2011}} for graphs of codimension one (see also \cite{deLimaADMmassasymptotically2014}) and by Mirandola and Vit\'orio~\cite{Mirandolapositivemasstheorem2015} for graphs of arbitrary codimension. The case of Euclidean hypersurfaces (not necessarily graphs), including the rigidity statement, was treated in \cite{Huanghypersurfaceswithnonnegativescalar2013}. Note that, since Euclidean graphs (of arbitrary codimension) and Euclidean hypersurfaces are spin, these cases also follow from Witten's proof.

In \cite{GeAnewmassforasymptoticallyflatmanifolds2014}, a new mass (actually, a family of masses) for asymptotically Euclidean manifolds, called the Gauss--Bonnet--Chern mass, was introduced. For a positive integer $q < n/2$, consider the $q$-th Gauss--Bonnet curvature, denoted $L_{(q)}$, and defined by
\begin{equation*} %\label{L}
  L_{(q)} = \frac{1}{2^q}
  \delta_{b_1 b_2 \ldots b_{2q - 1} b_{2q}}^{a_1 a_2 \ldots a_{2q - 1} a_{2q}}
  \prod_{s = 1}^{q} \tensor[]{R}{_{a_{2s - 1} a_{2s}}^{b_{2s - 1} b_{2s}}} = P_{(q)}^{ijkl}R_{ijkl},
\end{equation*}
where $R$ is the Riemann curvature tensor of $(M,g)$ and $P_{(q)}$, which has the same symmetries of~$R$ (see \cite[Section 3]{GeAnewmassforasymptoticallyflatmanifolds2014}), is given by
\begin{equation*} %\label{P}
 P_{(q)}^{ijkl} = \frac{1}{2^q}\delta_{b_1b_2 \cdots b_{2q-3}b_{2q-2}b_{2q-1}b_{2q}}^{a_1a_2 \cdots a_{2q-3}a_{2q-2}ij}
 \Biggl( \prod_{s=1}^{q-1} R\indices{_{a_{2s-1}a_{2s}}^{b_{2s-1}b_{2s}}} \Biggr)g^{b_{2q-1}k}g^{b_{2q}l}.
\end{equation*}

\begin{Definition}[GBC mass]
 The $q$-th GBC mass of an asymptotically Euclidean end $(E^n,g)$ is defined by
 \begin{equation} \label{GBC mass}
  \mathrm{m}_q(E,g) = c(n,q) \lim_{\rho \to \infty} \int_{S_\rho} P_{(q)}^{ijkl}g_{jk,l}\nu_i \,{\rm d}S_\rho,
 \end{equation}
 where
 \begin{equation*} %\label{eq:constant GBC mass}
  c(n,q) = \frac{(n-2q)!}{2^{q-1}(n-1)! \, \omega_{n-1}}
 \end{equation*}
 and $S_\rho$, ${\rm d}S_\rho$, $\nu$ and $\omega_{n-1}$ are as in the definition of the ADM mass.
\end{Definition}

Note that $L_{(1)}$ is just the scalar curvature and, as observed in \cite{GeAnewmassforasymptoticallyflatmanifolds2014}, $\mathrm{m}_1$ coincides with the ADM mass.

In the same article, it is shown that if $\tau > \tau_q$ and $L_{(q)}$ is integrable, then the limit (\ref{GBC mass}) exists, is finite, and is a geometric invariant, where here and throughout the text,
\begin{equation*} %\label{tau_q}
\tau_q = \frac{(n-2q)}{(q+1)}.
\end{equation*}

As in the $q=1$ case, if $(M^n,g)$ is an asymptotically Euclidean manifold of order $\tau > \tau_q$ whose $q$-th Gauss--Bonnet curvature is integrable, then its $q$-th GBC mass, denoted by $\mathrm{m}_q(M,g)$, is defined as the sum of the $q$-th GBC masses of its ends.

The following is a version of the PMT for the GBC mass.

\begin{Conjecture} \label{PMT for m_q}
Let $n$ and $q$ be integers such that $n \geq 3$ and $0 < q < n/2$. If $(M^n,g)$ is an asymptotically Euclidean Riemannian manifold of order $\tau > \tau_q$ whose $q$-th Gauss--Bonnet curvature $L_{(q)}$ is nonnegative and integrable, then the $q$-th GBC mass of each of its ends is nonnegative. Moreover, if the GBC mass of at least one of its ends is zero, then $(M,g)$ is isometric to the Euclidean space $(\mathbb{R}^n,\delta)$.
\end{Conjecture}

Conjecture \ref{PMT for m_q}, without the rigidity statement, was proved for graphs of codimension one in~\cite{GeAnewmassforasymptoticallyflatmanifolds2014}; the case of graphs of arbitrary codimension and flat normal bundle was done by Li, Wei and Xiong when $q=2$ \cite{LiGaussBonnetChernmassgraphic2014} and by the authors when $0 < q < n/2$ \cite{AlexandredeSousaUmBreveEstudos2016,AlexandredeSousaGaussBonnetChernmasshigher2019}. This conjecture, including the rigidity statement, is known to be true for conformally flat manifolds \cite{GeGaussBonnetChernmassconformally2014}.

\subsection{Mass in terms of the Lovelock tensor}

Let $(E^n,g)$, $n \geq 3$, be an asymptotically Euclidean end of order $\tau > \tau_1$, and let $G$ be its Einstein tensor, that is,
\begin{equation*}
  G = \mathrm{Ric} - \frac{1}{2}(\mathrm{Sc})g,
\end{equation*}
where $\mathrm{Ric}$ and $\mathrm{Sc}$ denote, respectively, the Ricci tensor and the scalar curvature of $(E,g)$. Throughout the text, we denote by $X$ the vector field given by{\samepage
\begin{equation} \label{X}
 X = x^i \frac{\partial}{\partial x^i},
\end{equation}
where $\Psi(x) = \bigl( x^1, x^2, \ldots, x^n \bigr)$ is a coordinate system satisfying (\ref{decay conditions}).}

Is is known that the ADM mass of $(E,g)$ can be computed as follows (see \cite{Ashtekar-Hansen,Chrusciel86,Chrusciel87}):
\begin{equation} \label{mass in terms of the Einstein tensor}
 \mathrm{m}_{\mathrm{ADM}}(E,g) = -\frac{1}{(n-1)(n-2)\omega_{n-1}} \lim_{\rho \to \infty} \int_{S_\rho} G(X,\nu_g) \, {\rm d}S_\rho^g,
\end{equation}
where $\omega_{n-1}$ and $S_\rho$ are as in (\ref{ADM mass expression}), $\nu_g$ is the outward unit normal vector to $S_\rho$ with respect to the metric $g$, and ${\rm d}S_\rho^g$ is the volume form induced on $S_\rho$ by $g$. The equivalence between formulas~(\ref{ADM mass expression}) and~(\ref{mass in terms of the Einstein tensor}) can be shown by reducing the general case to the case of harmonic asymptotics via a density theorem (see, for example, \cite{Huang12}). Proofs without the use of a density theorem were given in \cite{HerzlichComputingasymptoticinvariants2016,Miao-Tam}.

As we will see in a moment, a formula similar to (\ref{mass in terms of the Einstein tensor}) holds for the GBC mass. To state this result, we need to recall the so-called Lovelock curvature tensors.

Let $(M^n,g)$, $n \geq 3$, be a Riemannian manifold and let $q < n/2$ be a positive integer. The $q$-th Lovelock curvature tensor, denoted by $G_{(q)}$, is defined by
\begin{equation} \label{Lovelock tensor}
 G_{(q)ij} = - \frac{1}{2^{q + 1}} g_{ik}
 \delta_{j a_1 a_2 \ldots a_{2q - 1} a_{2q}}^{k b_1 b_2 \ldots b_{2q - 1} b_{2q}}
 \prod_{s = 1}^{q} \tensor[]{R}{_{b_{2s - 1} b_{2s}}^{a_{2s - 1} a_{2s}}}.
\end{equation}
Note that $G_{(1)}$ is just the Einstein tensor.
\begin{Proposition}
 The Lovelock curvature tensor $G_{(q)}$ satisfies the following:
 \begin{itemize}\itemsep=0pt
 \item[$(i)$] It is symmetric, that is,
  \begin{equation} \label{G is symmetric}
   G_{(q)ij} = G_{(q)ji}.
  \end{equation}
 \item[$(ii)$] It is divergent-free, that is,
  \begin{equation} \label{Lovelock divergent free}
   \nabla^iG_{(q)ij} = 0.
  \end{equation}
 \item[$(iii)$] Its trace satisfies the equation
  \begin{equation} \label{trace Lovelock}
   \mathrm{tr}_g \, G_{(q)} = - \frac{n-2q}{2} L_{(q)}.
  \end{equation}
 \end{itemize}
\end{Proposition}
\begin{proof}
   Identities (\ref{G is symmetric}) and (\ref{Lovelock divergent free}) follow from \cite[Theorem~1]{Lovelock71} (see also \cite{Lovelock70}). Identity (\ref{trace Lovelock}) is a~straightforward computation.
\end{proof}

Let $(E^n,g)$, $n \geq 3$, be an asymptotically Euclidean end of order $\tau > \tau_q$, where $q < n/2$ is a~positive integer. It was shown in \cite{WangWuChernsmagicform} that the $q$-th GBC mass of $(E,g)$ can be computed as follows:
\begin{equation} \label{GBC mass in terms of Lovelock tensor}
  \mathrm{m}_q(E,g) = -b(n,q) \lim_{\rho \to \infty} \int_{S_\rho} G_{(q)} (X, \nu_g) \, {\rm d}S_\rho^g,
\end{equation}
where $X$, $\nu_g$ and ${\rm d}S^g$ are as in (\ref{mass in terms of the Einstein tensor}) and
\begin{equation*}
  b(n,q) = \frac{(n - 2q - 1)!}{2^{q-1}(n-1)! \, \omega_{n-1}}.
\end{equation*}
Note that, when $q=1$, formulas (\ref{GBC mass in terms of Lovelock tensor}) and (\ref{mass in terms of the Einstein tensor}) coincide.

The following proposition establishes sufficient conditions for a vector field
to generate the GBC mass.
\begin{Proposition} \label{mass using another vector field}
   Let $(E^n,g)$, $n \geq 3$, be an asymptotically Euclidean end of order $\tau > \tau_q$, where $q < n/2$ is a positive integer. Let $\Psi(x) = (x_1, x_2, \ldots, x_n)$ be coordinates in $E$ satisfying $(\ref{decay conditions})$ and let $X$ be the vector field given by $(\ref{X})$.
  If $Y$ is a vector field on $E$ such that
  \begin{equation*}
   \bigl(Y^i - X^i\bigr)(x) = O\bigl(\rho^{-\tau + 1}\bigr), \qquad i = 1, 2, \ldots, n,
  \end{equation*}
  as $\rho \to \infty$, then
  \begin{equation*}
   \mathrm{m}_q(E,g) = -b(n,q) \lim_{\rho \to \infty} \int_{S_\rho} G_{(q)} (Y, \nu_g) \, {\rm d}S_\rho^g.
  \end{equation*}
\end{Proposition}
\begin{proof}
   Let $\Psi(x) = \bigl(x^1, x^2, \ldots, x^n\bigr)$ be coordinates on $E$ satisfying (\ref{decay conditions}). We can use these coordinates to compare $\nu_g$, the unit normal to $S_\rho$ with respect to $g$, and ${\rm d}S_\rho^g$, the volume form induced on $S_\rho$ by $g$, with their Euclidean counterparts $\nu_\delta$ and ${\rm d}S_\rho^\delta$, respectively.

  It holds
  \begin{equation*}
    \nu_g^i - \nu_\delta^i = O(\rho^{-\tau}),\qquad \text{as}\ \rho \to \infty,
  \end{equation*}
  and
  \begin{equation*}
    {\rm d}S_\rho^g = (1+w) \, {\rm d}S_\rho^\delta
  \end{equation*}
  for some function $w\colon E \to \mathbb{R}$ satisfying
  \begin{equation*}
    w = O(\rho^{-\tau}),\qquad \text{as} \ \rho \to \infty.
  \end{equation*}

  Furthermore, the components of the Riemman curvature tensor satisfy
  \begin{equation*}
    R^l_{ijk} = O\bigl(\rho^{-\tau - 2}\bigr), \qquad \text{as} \  \rho \to \infty.
  \end{equation*}
  Together with (\ref{Lovelock tensor}), this gives
  \begin{equation*}
  G_{(q)ij} = O\bigl(\rho^{-q(\tau + 2)}\bigr), \qquad \text{as} \  \rho \to \infty.
  \end{equation*}

  Thus, we have
  \begin{align*}
    &\biggl| \int_{S_\rho} G_{(q)}(Y,\nu_g) \, {\rm d}S_\rho^g - \int_{S_\rho} G_{(q)}(X,\nu_g) \, {\rm d}S_\rho^g \biggr| \\
     &\qquad{} \leq \int_{S_\rho} \bigl| G_{(q)ij}\bigl(Y^i - X^i\bigr)\bigr| \bigl(\bigl|\nu_g^i - \nu_\delta^i\bigr| + \bigl|\nu_\delta^i\bigr|\bigr)(1 + w)\, {\rm d}S_\rho^\delta \\
     &\qquad{} \leq C(n) \rho^{-[q(\tau + 2) + (\tau - 1) - (n-1)]}(1 + \rho^{-\tau})^2,
  \end{align*}
where $C(n)$ is a constant that depends only on the dimension.
  Using that $\tau > \tau_q$, it follows that%
  \begin{equation} \label{limit equals zero}
  \lim_{\rho \to \infty}  \biggl| \int_{S_\rho} G_{(q)}(Y,\nu_g) \, {\rm d}S_\rho^g -\int_{S_\rho} G_{(q)}(X,\nu_g) \, {\rm d}S_\rho^g \biggr| = 0.
  \end{equation}

The proposition follows from (\ref{GBC mass in terms of Lovelock tensor}) and (\ref{limit equals zero}).
\end{proof}

The specialization of this proposition to the case of gradient fields is a key ingredient
in the development of our extrinsic approach.
\begin{Corollary} \label{gradient field}
  Under the same hypothesis as the ones in Proposition $\ref{mass using another vector field}$, if there exists ${f \in C^{\infty}(E)}$ such that
\begin{equation} \label{decay gradient}
  \frac{\partial}{\partial x^i} \biggl( f - \frac{\rho^2}{2} \biggr) = O\bigl(\rho^{-\tau +1}\bigr), \qquad i = 1, 2, \ldots, n,
\end{equation}
as $\rho \to \infty$, then
\begin{equation*}
  \mathrm{m}_q(E,g) = -b(n,q) \lim_{\rho \to \infty} \int_{S_\rho} G_{(q)} \bigl(\nabla f, \nu_g\bigr) \, {\rm d}S_\rho^g.
\end{equation*}
\end{Corollary}
\begin{proof}
Write $g^{ij} = \delta^{ij} + \theta^{ij}$ and take $Y = \nabla f$. We have
\begin{align}
  Y^i - X^i& = g^{ij} \frac{\partial f}{\partial x^j} - \frac{\partial}{\partial x^i} \biggl( \frac{\rho^2}{2} \biggr)
  = \delta^{ij} \frac{\partial f}{\partial x^j} + \theta^{ij}\frac{\partial f}{\partial x^j} - \frac{\partial}{\partial x^i} \biggl( \frac{\rho^2}{2} \biggr) \nonumber \\
  & = \frac{\partial}{\partial x^i} \biggl( f - \frac{\rho^2}{2} \biggr) + \theta^{ij}\frac{\partial f}{\partial x^j}. \label{Xi - Yi}
\end{align}

Note that (\ref{decay gradient}) is equivalent to
\begin{equation} \label{decay 1}
  \frac{\partial f}{\partial x^i} = x^i + O\bigl(\rho^{-\tau + 1}\bigr), \qquad \text{as} \  \rho \to \infty.
\end{equation}
Thus, since
\begin{equation} \label{decay 2}
  \theta^{ij} = O( \rho^{-\tau}), \qquad \text{as} \  \rho \to \infty,
\end{equation}
from (\ref{decay 1}) and (\ref{decay 2}) we find
\begin{equation} \label{decay 3}
  \theta^{ij}\frac{\partial f}{\partial x^j} = O\bigl( \rho^{-\tau + 1} \bigr), \qquad \text{as} \  \rho \to \infty.
\end{equation}

The corollary follows from Proposition \ref{mass using another vector field}, (\ref{Xi - Yi}) and (\ref{decay 3}).
\end{proof}

\section{Imersions in Euclidean spaces}

Let $\psi\colon  M^n \looparrowright \mathbb{R}^d$, $d > n$, be a smooth immersion and
let \(\bar{\delta} = \langle \cdot , \cdot \rangle\) be the canonical Euclidean
metric on \(\mathbb{R}^d\). All quantities related to this ambient space will
also be denoted with an overbar, unless otherwise indicated.

Let $B$ (which is normal-vector valued) be the second fundamental form
of the immersion and let $p$ be an integer such that $0 < p \leq n$; if $p$
is even, then the $p$-th mean curvature $S_{(p)} \in C^\infty(M)$ is defined by
\begin{equation*}
   {S}_{(p)} = \frac{1}{p!} \delta_{a_1 \ldots a_p}^{b_1 \ldots b_p}
   \bigl\langle B_{b_1}^{a_1}, B_{b_2}^{a_2} \bigr\rangle \cdots \bigl\langle B_{b_{p-1}}^{a_{p-1}}, B_{b_p}^{a_p} \bigr\rangle,
\end{equation*}
and, if $p$ is odd, then the $p$-th mean curvature $S_{(p)} \in {\Gamma\bigl(TM^{\perp}\bigr)}$ is defined by
\begin{equation*}
   {S}_{(p)} = \frac{1}{p!} \delta_{a_1 \ldots a_p}^{b_1 \ldots b_p}
   \bigl\langle B_{b_1}^{a_1}, B_{b_2}^{a_2} \bigr\rangle \cdots \bigl\langle B_{b_{p-2}}^{a_{p-2}}, B_{b_{p-1}}^{a_{p-1}} \bigr\rangle B_{b_p}^{a_p},
\end{equation*}
where $\langle \cdot , \cdot \rangle$ denotes the Euclidean metric on $\mathbb{R}^d$.
We also set $S_{(0)} = 1$ and $S_{(n+1)} = 0$.

Next, we turn to the so-called Newton transformations. The \(0\)-th Newton transformation is defined as
\begin{equation*}
  T_{(0)} = g,
\end{equation*}
where \(g = \psi^\ast \bar{\delta}\) denotes the induced metric. If $p \geq 2$ is even,
then the $p$-th Newton transformation~$T_{(p)}$ is defined by
\begin{equation} \label{NT for even p}
T_{(p)ij} = \frac{1}{p!}g_{ik}\delta_{ja_1a_2\ldots a_p}^{kb_1b_2\ldots b_p}\bigl\langle B_{b_1}^{a_1}, B_{b_2}^{a_2} \bigr\rangle \cdots \bigl\langle B_{b_{p-1}}^{a_{p-1}}, B_{b_p}^{a_p} \bigr\rangle,
\end{equation}
and, if $p$ is odd, then the $p$-th Newton transformation $T_{(p)}$ is defined by
\begin{equation*}
T_{(p)ij} = \frac{1}{p!}g_{ik}\delta_{ja_1a_2\ldots a_{p-1}}^{kb_1b_2\ldots b_{p-1}}\bigl\langle B_{b_1}^{a_1}, B_{b_2}^{a_2} \bigr\rangle \cdots \bigl\langle B_{b_{p-2}}^{a_{p-2}}, B_{b_{p-1}}^{a_{p-1}}\bigr\rangle B_{b_p}^{a_p}.
\end{equation*}
Note that, when \(p\) is odd, the \(p\)-th Newton transformation is normal-vector
valued. Furthermore, by antisymmetry, it follows that \(T_{(p)} \equiv 0\)
when \(p > n\).

The relation between the trace of Newton transformations and the higher-order
mean curvatures is a fact well-known in the literature (see
\cite[Lemma 2.2]{GrosjeanUpperboundsfirst2002}, for example).
\begin{Proposition} %\label{traces}
  Let $p$ be an integer such that $0 \leq p \leq n$. The Newton transformation $T_{(p)}$ satisfy
  \begin{equation} \label{trace T}
    \mathrm{tr}_gT_{(p)} = (n-p)S_{(p)}.
  \end{equation}
 Moreover, if $p$ is even, then
 \begin{equation*} %\label{trace TB}
   T_{(p)ij}B^{ij} = (p+1)S_{(p+1)}.
 \end{equation*}
\end{Proposition}

By a direct application of the Gauss equation, we obtain that the \(q\)-th
Lovelock tensor (\ref{Lovelock tensor}) of the induced metric
\(g = \psi^\ast \bar{\delta}\) and the \(2q\)-th Newton transformation
(\ref{NT for even p}) of an immersion contain the same information.
\begin{Proposition} %\label{relation G and T}
  Let $q$ be a positive integer such that $q < n/2$. It holds
  \begin{equation} \label{equation relation G and T}
    G_{q} = - \frac{(2q)!}{2}T_{(2q)}.
  \end{equation}
\end{Proposition}

The next lemma is an infinitesimal version of a Pohozaev--Schoen-type integral
identity presented in \cite[Propositon 3.2]{deLimamasstermsEinstein2019}.
\begin{Lemma} %\label{div V interno K}
  If $K$ and $V$ are, respectively, a symmetric bilinear form and a vector field on $M$, then the following identity holds:
  \begin{equation} \label{equation div V interno K}
    \mathrm{div}_g (\iota_VK) = \iota_V(\mathrm{div}_gK) + \frac{1}{2}g(K, \mathcal{L}_Vg).
  \end{equation}
\end{Lemma}

Throughout the text, we denote by \(\bar{Z}\) the vector field on
$\mathbb{R}^d$ given by
\begin{equation*}
 \bar{Z} = \bar{x}^{\alpha} \frac{\partial}{\partial \bar{x}^{\alpha}},
 \qquad 1 \le \alpha \le d,
\end{equation*}
where \(\bigl( \bar{x}^1, \bar{x}^2, \ldots, \bar{x}^d \bigr)\)
is the standard coordinate system on \(\mathbb{R}^d\). It is known that
it is a conformal Killing gradient field and that it satisfies the following identity:
\begin{equation*}
 \bar{Z} = \frac{\overline{\nabla}\bar{\rho}^2}{2},
\end{equation*}
where \(\bar{\rho}\) is the distance function to origin on \(\mathbb{R}^d\).

The following proposition can be interpreted as an infinitesimal version
of the integral identity of Theorem \ref{main theorem}. It is inspired by a
infinitesimal version of the flux formula presented in~\mbox{\cite[equation (8.4)]{AliasConstanthigherordermean2006}}.

\begin{Proposition}
 Let $Y = \psi^\ast \bar{Z}^\top$, where $\bar{Z}^\top$ denotes the tangent part of the vector field \(\bar{Z}\). It~holds
  \begin{equation} \label{div Y interno G}
   \mathrm{div}_g(\iota_YG_{(q)}) = -\frac{(2q)!}{2} \bigl[ (n-2q)S_{2q} + (2q + 1) \bigl\langle S_{2q+1}, \bar{Z} \bigr\rangle \bigr].  \end{equation}
\end{Proposition}
\begin{proof}
  Since $Y = \psi^\ast \bar{Z}^\top$, we have
\begin{equation*}
  \psi_\ast Y = \bar{Z}^\top = \bar{Z} - \bar{Z}^\perp
\end{equation*}
and
\begin{align}
\mathcal{L}_Yg  = \mathcal{L}_Y\psi^\ast \bar{\delta}
 = \psi^\ast \bigl( \mathcal{L}_{\psi_\ast Y}\bar{\delta}\bigr)
 = \psi^\ast \bigl( \mathcal{L}_{(\bar{Z} - \bar{Z}^\perp)} \bar{\delta} \bigr)
 = \psi^\ast \bigl( \mathcal{L}_{\bar{Z}} \bar{\delta}\bigr) - \psi^\ast \bigl( \mathcal{L}_{\bar{Z}^\perp}\bar{\delta}\bigr). \label{Lie de g na direcao de Y}
\end{align}

Denote by $B_{\bar{Z}}$ the symmetric bilinear form on $M$ defined by
\begin{equation*}
 B_{\bar{Z}}(V,W) = \bigl\langle B(V,W), \bar{Z} \bigr\rangle.
\end{equation*}
For any $V,W \in \Gamma(TM)$,
\begin{align*}
 \bigl( \mathcal{L}_{\bar{Z}^\perp} \bar{\delta} \bigr) (V,W) = \bigl\langle D_V \bar{Z}^\perp, W \bigr\rangle + \bigl\langle V, D_W \bar{Z}^\perp \bigr\rangle = - 2 \bigl\langle B(V,W), \bar{Z} \bigr\rangle
  = - 2B_{\bar{Z}}(V,W),
\end{align*}
that is,
\begin{equation} \label{Lie derivative bar Z-perp}
 \mathcal{L}_{\bar{Z}^\perp} \bar{\delta} = -2B_{\bar{Z}}.
\end{equation}
We also have
\begin{equation} \label{Lie derivative bar Z}
 \mathcal{L}_{\bar{Z}}\bar{\delta} = 2\bar{\delta}.
\end{equation}
Thus, (\ref{Lie de g na direcao de Y})--(\ref{Lie derivative bar Z}) yield
\begin{equation} \label{produto interno de G com Lie de g na direcao de Y}
 g(G_{(q)}, \mathcal{L}_Yg) = 2 \, \mathrm{tr}_gG_{(q)} + 2g  (G_{(q)},B_{\bar{Z}} ).
\end{equation}

Identity (\ref{div Y interno G}) follows from (\ref{Lovelock divergent free}), (\ref{trace T})--(\ref{equation div V interno K}) and (\ref{produto interno de G com Lie de g na direcao de Y}).
\end{proof}

\subsection{Asymptotically Euclidean imersions}

Next, we describe our main object of study: a class of smooth immersions
that place certain smooth manifolds \(M^n\) into some Euclidean space
\(\mathbb{R}^d\), where \(d > n\), in a special way.

\begin{Definition}[asymptotically Euclidean immersion] \label{defn AE}
 Let \({\psi\colon M^n \looparrowright \mathbb{R}^d}\), \(d > n\), be a smooth
 immersion of a smooth manifold \(M^n\), \(n \geq 3\). We say that
 \(\psi\) is an asymptotically Euclidean immersion of order \(\tau\), for some
 \(\tau > 0\), if the following conditions hold:
\begin{itemize}\itemsep=0pt
  \item[(i)] the Riemannian manifold $\bigl(M,\psi^*{\bar{\delta}}\bigr)$ is complete and asymptotically Euclidean of order $\tau$;
  \item[(ii)] if $E \subset M$ is such that $(E,\psi^*{\bar{\delta}})$ is an asymptotically Euclidean end of order \(\tau\), then, as~${\rho \to \infty}$,
  \[\frac{\partial}{\partial x^i}\bigl( \psi^\ast\bar{\rho}^2 - \rho^2 \bigr)(x) = O\bigl(\rho^{-\tau+1}\bigr), \qquad i=1,2, \ldots, n,\]
  for any coordinate system $\Psi(x) = \bigl(x^1, x^2, \ldots, x^n\bigr)$ on $E$ satisfying (\ref{decay conditions}).
\end{itemize}
\end{Definition}

The following theorem is the main result of this article. It can be seen as a version of the Hsiung--Minkowski formulas \cite{Hsiungintegralformulasclosed1956,KwongExtensionHsiungMinkowski2016} to asymptotically Euclidean immersions.

\begin{Theorem} \label{main theorem}
 Let $n$ and $q$ be integers such that $n \geq 3$ and $0 < q < n/2$.
 If $\psi\colon  M^n \hookrightarrow \bigl(\mathbb{R}^d, \bar{\delta}\bigr)$, $d > n$, is an
 asymptotically Euclidean immersion of order $\tau > \tau_q$ such that
 \(L_{(q)}\), the $q$-th Gauss-Bonnet curvature of the induced metric
 \(\psi^{\ast}\bar{\delta}\), is integrable, then the functions $S_{2q}$
 and $\bigl\langle S_{2q+1},\bar{Z} \bigr\rangle$ are integrable
 and the $q$-th Gauss--Bonnet--Chern mass of $\bigl(M, \psi^\ast\bar{\delta}\bigr)$ satisfies
 the following integral identity:
 \begin{equation*}
  \mathrm{m}_q\bigl(M,\psi^*\bar{\delta}\bigr) = a(n,q) \biggl[
   (n-2q) \int_M S_{2q} \, {\rm d}M
   + (2q+1) \int_M \langle S_{2q+1}, \bar{Z} \rangle \, {\rm d}M \biggr],
 \end{equation*}
 where \({\rm d}M\) is the Riemannian measure on \(\bigl(M, \psi^{\ast}\bar{\delta}\bigr)\) and
 \begin{equation*}
  a(n,q) = \frac{(2n)!(n-2q-1)!}{2^q(n-1)! \, \omega_{n-1}}.
 \end{equation*}
\end{Theorem}

\begin{proof}
 Let $E$ be one of the ends of $M$. Consider the function $f\colon M \to \mathbb{R}$, with
 \begin{equation*}
  f = \frac{\psi^\ast \bar{\rho}^2}{2}.
 \end{equation*}
 Since $\psi\colon  M^n \looparrowright \mathbb{R}^d$ is an asymptotically Euclidean immersion of order $\tau > \tau_q$, the function $f$ satisfies (\ref{decay gradient}), and hence, by Corollary \ref{gradient field},
 \begin{equation} \label{eq 1}
  \mathrm{m}_q \bigl(E,\psi^\ast \bar{\delta}\bigr) = -b(n,q) \lim_{\rho \to \infty} \int_{S_\rho} G_{(q)}(\nabla f,\nu)\, {\rm d}S_\rho,
 \end{equation}
 where the geometric quantities in the integral are computed with respect
 the induced metric~\(\psi^\ast \bar{\delta}\).

 Note that
 \begin{equation*}
  \nabla f = \psi^\ast \bar{Z}^\top.
 \end{equation*}
 Thus, applying equation (\ref{eq 1}) to all the ends of $M$ together
 with the divergence theorem and identity (\ref{div Y interno G}), we find
 \begin{equation} \label{eq 2}
  \mathrm{m}_q\bigl(M,\psi^\ast \bar{\delta}\bigr)
  = a(n,q) \int_M \bigl[ (n-2q)S_{2q} + (2q + 1) \bigl\langle S_{2q+1}, \bar{Z} \bigr\rangle \bigr] {\rm d}M.
 \end{equation}

 It remains to show that both $S_{2q}$ and $\bigl\langle S_{2q+1}, \bar{Z} \bigr\rangle$ are integrable, so the right hand side of~(\ref{eq 2}) can be broken in two. The integrability of $S_{2q}$ follows from (\ref{trace Lovelock}), (\ref{trace T}), (\ref{equation relation G and T}) and the integrability of $L_{(q)}$. Once we know that $S_{(2q)}$ is integrable, the integrability of $\bigl\langle S_{2q+1}, \bar{Z} \bigr\rangle$ follows from the fact that the left hand side of (\ref{eq 2}) is finite.
\end{proof}

An immediate consequence of this theorem is the following.

\begin{Corollary} \label{corollary main theorem}
 Under the same hypothesis as the ones in Theorem~{\rm \ref{main theorem}},
 $\mathrm{m}_q\bigl(M,\psi^*\bar{\delta}\bigr) \geq 0$ if and only if
 \begin{equation*}
  (n-2q) \int_M S_{2q} \, {\rm d}M + (2q+1) \int_M \bigl\langle S_{2q+1}, \bar{Z} \bigr\rangle \, {\rm d}M \geq 0.
 \end{equation*}
\end{Corollary}

A famous theorem by Nash \cite{Nash} states that any Riemannian manifold $(M,g)$ can be isometrically immersed in some Euclidean space $\bigl(\mathbb{R}^d, \bar{\delta}\bigr)$. Encouraged by Nash's theorem, we make the following conjecture.

\begin{Conjecture} \label{q1}
   If $(M^n,g)$, $n \geq 3$, is an asymptotically Euclidean manifold of order $\tau > 0$, then there exists an asymptotically Euclidean isometric immersion $\psi\colon  M \to \mathbb{R}^d$ of order $\tau$
  $($in the sense of Definition $\ref{defn AE})$.
\end{Conjecture}

Let $n$ and $q$ be integers such that $n \geq 3$ and $0 < q < n/2$. Suppose that Conjecture \ref{q1} is true for every asymptotically Euclidean manifold $(M^n,g)$ of order $\tau > \tau_q$ (or at least for those whose Gauss--Bonnet curvature $L_{(q)}$ is nonnegative and integrable). Then, by Theorem \ref{main theorem}, Conjecture~\ref{PMT for m_q} would be a direct consequence of the following conjecture.

\begin{Conjecture} \label{q2}
   Let $n$ and $q$ be integers such that $n \geq 3$ and $0 < q < n/2$. Let $\psi\colon  M^n \to \mathbb{R}^d$, $d>n$, be an asymptotically Euclidean immersion of order $\tau > \tau_q$ for which $S_{2q}$ is integrable and nonnegative. It holds
  \begin{equation*}
    (n-2q) \int_M S_{2q} \, {\rm d}M + (2q+1) \int_M \bigl\langle S_{2q+1}, \bar{Z} \bigr\rangle \, {\rm d}M \geq 0,
  \end{equation*}
  with the equality holding if and only if $\bigl(M,\psi^\ast \bar{\delta}\bigr)$ is isometric
  to Euclidean space.
\end{Conjecture}

\subsection*{Acknowledgements}

This study was partially financed by the Coordena\c c\~ao de Aperfei\c coamento de Pessoal de N\'ivel Superior - Brasil (CAPES) - Finance Code 001.
This work was partially done while Alexandre de Sousa was a CAPES Fellow at the Mathematics Institute of Federal University of Alagoas (IM/UFAL), whose members he would like to thank for the hospitality.
Frederico Gir\~ao was partially supported by CNPq, grant number 307239/2020-9.

%\bibliographystyle{sigma}
%\bibliography{references}

\pdfbookmark[1]{References}{ref}
\LastPageEnding

\end{document}